\documentclass[11pt]{amsart}
\usepackage{afterpage}
\usepackage{graphicx, epstopdf, color, verbatim} 
\usepackage{amsmath,amsfonts,amsthm,stmaryrd,amssymb} 
\usepackage{enumitem}
\usepackage[pdftex,pagebackref,letterpaper=true,colorlinks=true,pdfpagemode=none,urlcolor=blue,linkcolor=blue,citecolor=blue,pdfstartview=FitH]{hyperref}

\usepackage{tikz,pgf}
\usetikzlibrary{positioning, arrows,topaths}

\newtheorem{theorem}{Theorem}
\newtheorem{lemma}[theorem]{Lemma}
\newtheorem{corollary}[theorem]{Corollary}
\newtheorem{proposition}[theorem]{Proposition}
\newtheorem{conjecture}[theorem]{Conjecture}
  \theoremstyle{definition}
\newtheorem{definition}[theorem]{Definition}
\newtheorem{example}{Example}
\newtheorem{remark}[theorem]{Remark}

\def\E{\mathcal{E}}
\def\Z{\mathbb{Z}}
\def\R{\mathbb{R}}

\def\GS{\mathcal{GSV}}
\def\MBV{\mathcal{MBV}}
\DeclareMathOperator{\poly}{poly}

\DeclareMathOperator{\infconv}{\square_\mathbb{Z}}

\DeclareMathOperator{\conv}{conv}

\begin{document}
\title[Finite Matroid-Based Valuation Conjecture is False]{\bf The Finite Matroid-Based Valuation Conjecture is False}
\author{Ngoc Mai Tran}
\address{Department of Mathematics, University of Texas at Austin, TX 78712}
\email{ntran@math.utexas.edu}
\thanks{The author would like to thank Rakesh Vohra for introducing the problem, and Renato Paes Leme and Kazuo Murota for helpful feedback on an earlier version of the manuscript. The author is very grateful for two anonymous referees for their careful reading and helpful comments, including a much shorter proof of Lemma \ref{lem:in.gs}.}
\date{\today}

\maketitle


\begin{abstract}
The matroid-based valuation conjecture of Ostrovsky and Paes Leme \cite{ostrovsky2015gross} states that all gross substitutes valuations on $n$ items can be produced from merging and endowments of weighted ranks of matroids defined on at most $m(n)$ items. We show that if $m(n) = n$, then this statement holds for $n \leq 3$ and fails for all $n \geq 4$. In particular, the set of gross substitutes valuations on $n \geq 4$ items is strictly larger than the set of matroid based valuations defined on the ground set $[n]$. Our proof uses matroid theory and discrete convex analysis to explicitly construct a large family of counter-examples. 
It indicates that merging and endowment by themselves are poor operations to generate gross substitutes valuations. 
We also connect the general MBV conjecture and related questions to long-standing open problems in matroid theory, and conclude with open questions at the intersection of this field and economics. 
\end{abstract}
\section{Introduction}
Gross substitutes valuations form an algorithmically tractable subclass of submodular functions on $2^{[n]}$ to $\R$. They are of special interest to combinatorial auctions \cite{kelso1982job,gul1999walrasian,danilov2001discrete,ausubel2002ascending,reijnierse2002verifying,bing2004presentation,hatfield2005matching,lehmann2006combinatorial}, have numerous applications and have been discovered and rediscovered in various contexts: matroid theory and optimization \cite{edmonds1970submodular,dress1990valuated}, algebraic geometry \cite{gelfand1987combinatorial,kapranov1993chow,keel2006geometry,speyer2008tropical,ben2017walrasian}, and discrete convex analysis \cite{murota1999m,murota2003discrete}, see \cite{leme2017gross} for a comprehensive recent survey. Kelso and Crawford \cite{kelso1982job} put forward the notion of gross substitutes as a way to generalize the theory of pricing and ascending auctions that had been developed earlier for matching markets. When agents valuations are gross substitutes, competitive equilibrium is guaranteed to exist \cite{bikhchandani1997competitive} and the competitive prices can be found by a greedy algorithm \cite{gul1999walrasian}. From the agents' viewpoint, however, specifying an arbitrary gross substitutes function on $[n]$  requires at least $2^{n/\poly(n)}$ values \cite{hajek2008substitute, leme2017gross}. This presents a major practical difficulty in implementing combinatorial auctions with gross substitutes.

A number of papers have been devoted to finding constructive characterizations for gross substitutes valuations \cite{hatfield2005matching, hajek2008substitute,kojima2014designing,ostrovsky2015gross,leme2017gross, milgrom2017discovering, balkanski2018construction}.
The general idea is to start with a class of known gross substitutes valuations, and close it up under operations that preserve gross substitutability. Two natural operations with simple economics interpretations are merging and endowments. With these operations, Hatfield and Milgrom \cite{hatfield2005matching} proposed to start with unit demand valuations and called the resulting class endowed assignment valuations (EAVs). They showed that this family encompasses a large number of gross substitutes valuations frequently used in economics. Ostrovsky and Paes Leme proved that not all gross substitutes valuations are EAVs \cite{ostrovsky2015gross}. They proposed to start with a richer class: weighted ranks of all matroids on a finite set $[m]$. The resulting class, matroid-based valuations (MBVs), is conjectured to be equal to the set of gross substitutes. 

To make this conjecture precise, one needs to clarify the relation between $m$, the number of items that the matroids in the generating set are defined on, and $n$, the number of items in the target class of gross substitutes valuations. Let $\GS_n$ be the set of all gross substitutes valuations whose ground set is some subset of $[n]$
$$ \GS_n = \{u: 2^{S} \to \R_{\geq 0}, S \subseteq [n], u \mbox{ is a gross substitute valuation}\}. $$ 
Let $\mathcal{MBV}_{m,n}$ denote the smallest subset of $\GS_n$ that (a) is closed under merging and endowment, and (b) contains weighted ranks of all matroids whose ground set is a subset of $[m]$. The MBV conjecture reads as follows.
\footnote{We note that a literal translation of the statement of Otrovsky and Leme allows $m$ to depends on $v$ instead of $n$, that is, $v \in \MBV_{m(v),n}$. The only case where this distinction matter is when $\sup_{v \in \GS_n}m(v) = \infty$, that is, $m(n) = \infty$. In other words, to represent all gross substitutes valuations on $n$ items, one would need to start with \emph{all} weighted ranks of \emph{all} possible matroids. Even if such a result was true, it would not be a helpful characterization of gross substitutes. Therefore, we exclude this case in the formulation of Conjecture~\ref{conj:mbv}.} 

\begin{conjecture}[The Matroid-based Valuation Conjecture \cite{ostrovsky2015gross}]\label{conj:mbv}
For each $n \geq 1$, 
for all $v \in \GS_n$, there exists some integer $m(n)$ such that $v \in \mathcal{MBV}_{m(n),n}$. In other words, $\GS_n \subseteq \MBV_{m(n),n}$.
\end{conjecture}

The larger $m(n)$ is, the more complex the starting class of weighted ranks one must start with. If $m(n) > n$, then the generators of $\MBV_{m(n),n}$ contains functions which are not in $\GS_n$, and thus the less attractive it is to represent gross substitutes valuations as matroid based valuations. Therefore, it is important to establish a lower-bound for $m(n)$, and in particular, to verify this conjecture for the case $m(n) = n$. Our main result completely characterizes the relationship between $\MBV_{n,n}$ and $\GS_n$ for all $n$.

\begin{theorem}\label{thm:main}
For $n \leq 3$, $\MBV_{n,n} = \GS_n$. For $n \geq 4$, $\MBV_{n,n}~\subsetneq~\GS_n$. 
\end{theorem}

For $n \geq 4$, we introduce a new family called partition valuations, and show through direct calculations that they are in $\GS_n$ but not in $\MBV_n$. For the cases $n = 2$ and $n = 3$, our proof uses the combinatorial characterization of gross substitutes (cf. Theorem \ref{thm:murota}) to breakup $\GS_n$ into finitely many cones, each indexed by a particular \emph{regular subdivision} of $[0,1]^n$. We then give an explicit decomposition of valuations in each cone as a merge of unit demands. 

If weighted matroid ranks on $[n]$ cannot generate $\GS_n$ through merging and endowment, then what is the minimal set of generators? The second half of our paper collects various partial results to address this question. We show that the set of minimal generators contains the set of valuations defined on $2^{[n]}$ which are \emph{irreducible} under merging. We characterize the set of irreducibles amongst matroid-based valuations. It shows that merging is strongly tied to matroid union. 

\begin{theorem}\label{thm:ce.irreducible}
A weighted rank valuation on $[n]$ is irreducible with respect to merging if and only if the corresponding matroid is irreducible with respect to union. 
\end{theorem}

Characterizing irreducible matroid is a long-standing open problem posed by Welsh \cite[Problem 12.3.9]{oxley2006matroid}. General solutions are known only for binary matroids \cite{lovasz1973sum,cunningham1979binary,dawson1985decomposition,recski1989some}. Even if one accepts matroid irreducibility as a blackbox criterion, Theorem \ref{thm:main} says that their weighted ranks do not exhaust the set of all irreducibles. It follows from the proof of Theorem \ref{thm:main} that the partition valuations are also irreducible. For $n \geq 6$, we show that there are yet more irreducibles.

\begin{theorem}\label{thm:main2}
For $n \geq 6$, there exists irreducible gross substitutes valuations in $\GS_n$ that cannot be obtained from repeated merging and endowment of weighted ranks of matroids and partition valuations defined on ground sets $S \subseteq [n]$. 
\end{theorem}

The proof of Theorem \ref{thm:main2} constructs a large class of irreducible valuations from the regular subdivision of $[0,1]^n$ induced by rank functions of irreducible matroids. It gives important insights on the geometry of the merging operation. Intuitively, merging tend to produce 'smoother' functions with larger regions of lineality (cf. Lemma \ref{lem:gs.mixed}). This means any gross substitutes valuation with small regions of lineality tend to be irreducible. Furthermore, merging is not a local operation, so small local changes in irreducibles create more irreducibles (cf. Proposition \ref{prop:modify}). These observations indicate that merging is not a rich enough operations to generate $\GS_n$ from a small subset of valuations in $\GS_n$.

The endowment operation kicks in when we consider $m(n) > n$, as it allows one to merge functions defined on $2^m$ parameters before restricting down to a subset of $2^n$ values. In particular, thanks to an anonymous referee, we learned that each partition valuation is in $\MBV_{n+r,n}$ for some integer $r$ between $1$ and $n/2$. 
It remains an open problem whether the irreducibles of Theorem \ref{thm:main2} are in $\MBV$. Based on the geometry of merging discussed above, we speculate that the MBV conjecture is unlikely to hold, and even if it does, $m(n)$ may be very large. Going forward with constructive descriptions of gross substitutes, one may want to supplement merging and endowment with other operations such as matroid rank sums \cite{shioura2015gross} or tree-concordant sum \cite{balkanski2018construction}.

\subsection*{Organization} Section \ref{sec:background} presents the combinatorial view of gross substitutes and the connections between merging and matroid union. The proofs of the three theorems are presented in Sections \ref{sec:main} to \ref{sec:main2}. 
We conclude with a brief summary in Section~\ref{sec:summary}.
\vskip12pt
\noindent \textbf{Notations.} For a set $Q \subset \R^n$, $\conv(Q)$ denote its convex hull. A lattice polytope, also known as an integer polytope, is a polytope whose vertices are in $\Z^n$. For an integer $n$, let $[n] = \{1,2,\dots,n\}$. By an abuse of notation we use the same notation for both a subset $S \subseteq [n]$ and its indicator vector $e_S \in \{0,1\}^n$, trusting that no confusions will arise. Definition of matroid terminologies can be found in \cite{oxley2006matroid}.

\section{Gross substitutes as generalization of matroid ranks}\label{sec:background}
Fix an integer $n \in \mathbb{N}$. A function $u: 2^{[n]} \to \R_{\geq 0}$ is a \emph{valuation} if $u(\emptyset) = 0$ and $u(S) \leq u(T)$ if $S \subseteq T$. We shall extend the domain of $u$ to $\mathbb{Z}^n$ by defining $u(a) = -\infty$ for $a \notin \{0,1\}^n$. The Legendre-Fenchel transform of $u$ is the function $u^\vee: \R^n \to \R$ given by
\begin{equation}\label{eqn:f}
u^\vee(p) = \sup_{a \in \mathbb{Z}^n}(u(a) - \langle p, a \rangle) = \max_{a \in \{0,1\}^n}(u(a) - \langle p,a \rangle). 
\end{equation}
In an economy with $n$ indivisible objects, an agent with valuation $u$ has indirect utility $u^\vee$, which takes a price vector $p \in \R^n$, and maps it to the best utility that she can make under this price. 
For each $p \in \R^n$, the \emph{cell} or \emph{demand set} $\sigma \subset \{0,1\}^n$ supported by $p$ is the set of $a \in \{0,1\}^n$ that achieves the maximum in \eqref{eqn:f}. Its convex hull $\conv(\sigma) \subset [0,1]^n$ is called a \emph{face} supported by $p$. The union of faces supported by $p$ over all $p \in \R^n$ fits together to form a polyhedral complex called the regular subdivision of $[0,1]^n$ induced by $u$, denoted $\Delta_u$ \cite[\S 2.3]{detriangulations}. Here we emphasize that a cell is a set of integer points, while a face is the convex hull of such points and thus is a convex polytope. This distinction is important in economics applications, see \cite{danilov2001discrete,danilov2003gross, tran2015product}. 

Valuations with the \emph{gross substitutes} property enjoy many desirable properties for economics applications and thus have been intensively studied \cite{kelso1982job,gul1999walrasian,danilov2001discrete,ausubel2002ascending,reijnierse2002verifying,bing2004presentation,hatfield2005matching,lehmann2006combinatorial,leme2017gross,balkanski2018construction,tran2015product,baldwin2019understanding}. Since Kelso and Crawford's original definition \cite{kelso1982job}, many equivalent characterizations have been found. For a very recent and comprehensive survey, we recommend \cite{leme2017gross}. 
The characterization most relevant for our approach is in terms of the regular subdivisions $\Delta_u$. Since this is not listed in \cite{leme2017gross}, we include a short proof. Effectively, we take this characterization to be the definition of gross substitutes. 

\begin{definition}
Say that a lattice polytope $P \subseteq \R^n$ is $M^\natural$ if its edges are parallel to one of the vectors in $\{e_i-e_j,e_i: i,j \in [n], i \neq j\}$. 
\end{definition}

\begin{theorem}\label{thm:murota}
A valuation $u: \{0,1\}^n \to \R_{\geq 0}$ is gross substitutes if and only if all faces of the regular subdivision $\Delta_u$ are $M^\natural$ polytopes.
\end{theorem}
\begin{proof}
By  \cite[Theorem 17.1]{Fujishige05}\footnote{Fujishige attributed this result to an unpublished result of Tomizawa in 1983 \cite[p. 332]{Fujishige05} and proved in \cite{fujishige2003note}. The same result was also independently discovered by Danilov, Koshevoy and Lang \cite{danilov2003gross} and Gelfand, Goresky, MacPherson and Serganova \cite{gelfand1987combinatorial}.}, $u$ is gross substitutes if and only if it is an $M^\natural$-concave function on $\{0,1\}^n$. By Murota \cite[Theorem 6.30]{murota2003discrete}, this happens if and only if all faces of the regular subdivision $\Delta_{u}$ are $M^\natural$ polytopes.
\end{proof}


\begin{example}[Weighted matroid rank]
Fix a matroid $\E = ([n],\mathcal{I})$ with independence sets $\mathcal{I}$. The weighted matroid rank, or weighted matroid valuation $\rho^w: 2^{[n]} \to \R_{\geq 0}$ with weight $w \in \R^n_{\geq 0}$ is defined by
\vskip-0.3cm
\begin{equation}\label{eqn:weighted.matroid.val}
\rho^w(T) := \max\{\sum_{i \in I}w_i: I \in \mathcal{I}, I \subseteq T\}. 
\end{equation}
\vskip-0.1cm \noindent
View $\rho^w$ as a function from $\{0,1\}^n \to \R$. Then $\rho^w$ is gross substitutes \cite[equation (2.78)]{murota2003discrete}. When $w$ is the all-one vector $\mathbf{1}$, $\rho^{\mathbf{1}}$ is the rank function of $\E$.
\end{example}


\subsection{$\GS_n$ as a generalization of matroid rank functions} 
In particular, matroid union, contraction, deletion and duality are operations that produce new matroid ranks from old, and these can all be generalized into operations on valuations that preserve the gross substitutes property. The first two operations are of particular interests to economics, and they are known as merging and endowment, respectively.   

\begin{definition}[Merging and endowment]
For subsets $E^1,E^2 \subseteq [n]$, the merge of valuations $u^1: 2^{E^1} \to \R_{\geq 0}$ and $u^2: 2^{E^2} \to \R_{\geq 0}$ is the valuation $u^1 \ast u^2: 2^{E^1 \cup E^2} \to [n]$ given by
$$ u^1 \ast u^2(S) := \max_{T \subseteq S} (u(T) + v(S\backslash T)) \mbox{ for all } S \subseteq E^1 \cup E^2. $$
The endowment of a valuation $u: 2^{[n]} \to \R_{\geq 0}$ by a subset $T \subseteq [n]$ is the valuation $\partial^Tu: 2^{[n]\backslash T} \to \R_{\geq 0}$, given by
$$ \partial^Tu(S) = u(S\cup T) - u(T) \mbox{ for all } S \subseteq [n] \backslash T. $$
\end{definition}
In economics terms, $u \ast v$ is the valuation of a company formed by the merge of two agents with valuations $u$ and $v$, respectively. The endowment by $T$ is the valuation of an agent who started with $T$, so it measures how much another subset of items $S$ adds to what she already had. In the literature, merging is also called convolution \cite[\S 6]{murota2003discrete} and endowment is marginal valuation \cite{ostrovsky2015gross}. Both of these operations preserve gross substitutability \cite[Theorem 6.15]{murota2003discrete}. On matroid ranks, merging corresponds to matroid union and endowment corresponds to matroid contraction. 
\begin{lemma}
Consider matroids $\E^1 = (E^1,\mathcal{I}^1)$ and $\E^2 = (E^2,\mathcal{I}^2)$ with rank functions $\rho^1$ and $\rho^2$ respectively. Then $\rho^1 \ast \rho^2$ is the rank function of the matroid union $\E^1 \vee \E^2$, which is defined to be the matroid on $E^1 \cup E^2$ with independence sets
$$ \mathcal{I} = \{I^1 \cup I^2: I^1 \in \mathcal{I}^1,I^2 \in \mathcal{I}^2\}. $$
For $T \subseteq E^1$, the endowment $\partial^{T}u$ is the rank function of the matroid obtained as the contraction of $\E^1$ to $T$. 
\end{lemma}
\begin{proof}
The first result is due to \cite{pym1970submodular}, see also \cite[p410,problem 8]{oxley2006matroid}. The second statement is \cite[Proposition 3.1.6]{oxley2006matroid}.
\end{proof}


\begin{example}[OXS, EAV, transversal matroids and gammoids]\label{ex:unit}
A \emph{unit demand} valuation has the form $u: 2^{[n]} \to \R_{\geq 0}$ by
$$ u(T) = \max_{i \in T} w_i $$
for some $w \in \R^n_{\geq 0}$. It is also the weighted rank of the uniform matroid $U_{1,S}$ of rank 1 on the ground set $[n]$. Such a valuation can represented graphically as a bipartite graph with $[n]$ nodes on the left, labelled $1$ to $n$, and one node $A$ on the right, with $w_i$ the weight of the edge $(1,A)$. The class of valuations generated by merging finitely many unit demands is called OXS \cite{lehmann2006combinatorial} or max bipartite matching valuations \cite[\S 2]{murota2003discrete}. For concrete examples, see Figures \ref{fig:type0} to \ref{fig:type.last}. 
The EAV class generalizes OXS by allowing for repeated applications of merging and endowment. 
When the weights that define the individual unit demands are binary vectors, then their merge is the rank of a transversal matroid \cite[Proposition 12.3.7]{oxley2006matroid}, while the closure of transversal matroids under endowment is gammoids \cite[Proposition 3.2.10]{oxley2006matroid}. Thus, OXS generalizes transversal matroids, while EAV generalizes gammoids. 
\end{example}

\subsection{Minimal generators and irreducibles}
Since not all matroids are gammoids \cite[p411, exercise 11]{oxley2006matroid}, Remark \ref{ex:unit} hints that not all EAV are GSV. Indeed, the counter-example constructed by Paes Leme and Ostrovsky relied on the rank of function of $M(K_4)$, a matroid that is not a gammoid. On the other hand, the generating set of $\MBV_{n,n}$ contains all weighted matroid ranks (and hence all matroid ranks) of all matroids whose ground set is a subset of $[n]$. Thus it is less clear why $\MBV_{n,n}$ fails to equal to $\GS_n$ for $n \geq 4$. 

Let $\mathcal{G}_n \subset \GS_n$ be the minimal set of generators for $\GS_n$ under merging and endowment. 
Then $\mathcal{MBV}_{n,n} = \GS_n$ if and only if $\mathcal{G}_n$ is contained in the set of all weighted matroid ranks on $[n]$. 
The key idea of our paper is to identify a subset of $\mathcal{G}_n$ that is easy to work with. These are irreducible gross substitute valuations with ground set $[n]$, and they are simple because we do not need to consider endowment. In particular, they generalize the rank functions of matroids which are irreducible under merging. 
\begin{definition}
Let $v \in \GS_n$. Say that $v$ is irreducible under merging in $\GS_n$ (abbreviated irreducible) if $v = u \ast u'$ for some $u,u' \in \GS_n$ implies either $v = u$ or $v = u'$.  
\end{definition}

\begin{corollary}\label{rem:irr}
If $v \in \GS_n$ has ground set $[n]$, then $v \in \mathcal{G}_n$ if and only if it is irreducible. 
\end{corollary}
\begin{proof}
Since $v$ has ground set $[n]$, $v$ cannot equal $\partial^Tu$ for some $u \in \GS_n$ and $T \subsetneq [n]$. Thus $v$ is can only be generated from other functions in $\GS_n$ through merging. Thus it is in the set of minimal generators if and only if it is irreducible.
\end{proof}

\section{Proof of Theorem \ref{thm:main}}\label{sec:main}

\subsection{Proof of Theorem \ref{thm:main} for $n \leq 3$}
Theorem \ref{thm:murota} implies that there are finitely many combinatorial type of gross substitute valuations indexed by the admissible regular subdivisions.
Each combinatorial type corresponds to a cone defined by linear inequalities. The proof for $n \leq 3$ enumerates all such possible regular subdivisions $\Delta_v$ for $v \in \GS_n$, and gives an explicit decomposition of $v$ in each case as a merge of unit demands (and thus are in $\MBV_n$). The case $n = 1$ is trivial. For $n=2$, there are two combinatorial types and their decompostions are given in Figure \ref{fig:n2}. For $n=3$, up to permutation there are 8 combinatorial types, and their decompositions are given in Figures \ref{fig:type0} to \ref{fig:type.last}. \qed
\begin{figure}[h!]
\includegraphics[width=0.6\textwidth]{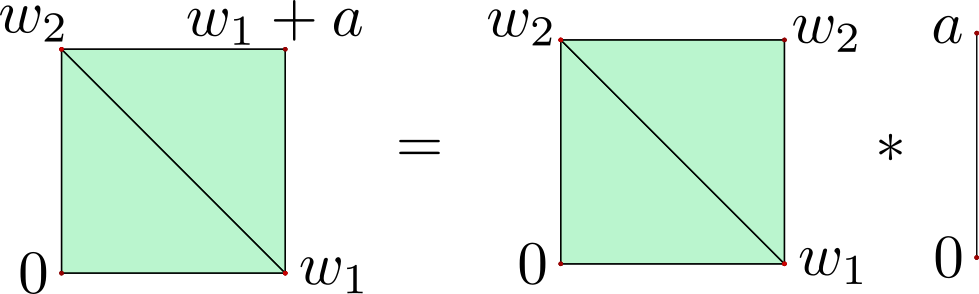} 
\vskip0.5cm
\includegraphics[width=0.6\textwidth]{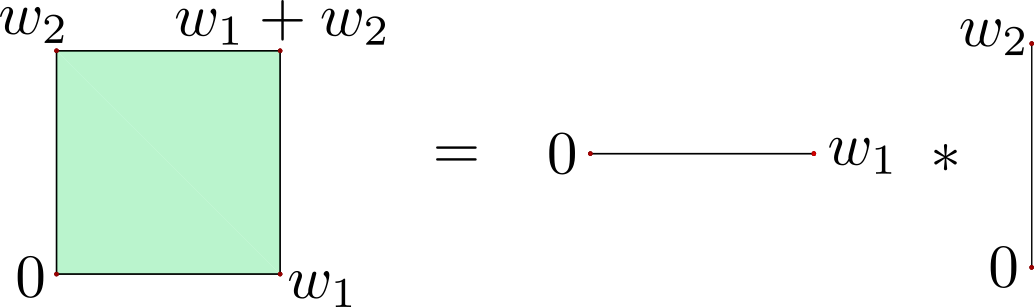} 
\caption{A short proof that $\GS_2 = \MBV_2$. There are two possible combinatorial types for valuations in $\GS_2$. In the top case, $0 \leq w_1 \leq w_2 \leq w_1 + a < w_1+w_2$. In the bottom case, $0 \leq w_1, w_2$.}\label{fig:n2}
\end{figure}


\begin{corollary}\label{cor:oxs}
For $n \leq 3$, any gross substitutes valuation $v \in \GS_n$ can be written as the convolution of unit-demand valuations. That is, for $n \leq 3$, the classes gross substitute valuations, matroid-based valuations, endowed assignment valuations and OXS are all equal.
\end{corollary}

\subsection{The case $n \geq 4$}\label{sec:clique}

For each $n \geq 4$, to show that $\MBV_n \subsetneq \GS_n$, we shall exhibit functions in the set of minimal generators $\mathcal{G}_n$ which are not weighted matroid ranks. These functions are chosen from a parametrized family indexed by partitions on $[n]$. 

\begin{definition}
Let $\pi = \{\pi_1, \dots, \pi_m\}$ be a set partition of $[n]$ into $m$ parts. For $0 < a < b < 2a$, the partition valuation $v = v^{\pi,a,b}$ on $[n]$ is given by
\begin{align*}
v(\emptyset) &= 0\\
v(I) &= 
\left\{ 
\begin{array}{ccc}
a & \mbox{ if } & I \subseteq \pi_r \mbox{ for some } r \in \{1,\dots,m\} \\
b & \mbox{ else.}  &
\end{array}\right. 
\end{align*}
\end{definition}
In other words, we think of $\mathcal{S}$ as partitioning the complete graph on $[n]$ vertices into $m$ cliques (complete subgraphs). Each $I \subseteq [n]$ identifies a subgraph, and $v(I) = b$ if this subgraph contains an edge that goes between two different cliques. Otherwise, $v(I) = a$. Note that $v(I) = a$ when $I$ is a singleton. 

\begin{lemma}\label{lem:in.gs}
For $n \geq 4$, for any partition $\pi$ of $[n]$ into $m$ parts and any $0 < a < b < 2a$, $v^{\pi,a,b} \in \GS_n$. 
\end{lemma}
\begin{proof}
The author is grateful for the following simple proof due to an anonymous referee. We shall prove that $v^{\pi,a,b} = \partial^{[n]}\bar{v}$ where $\bar{v}$ is the maximum bipartite valuation on a ground set of cardinality $n+m$, defined via Figure \ref{fig:oxs}.
\begin{figure}[h]
\includegraphics[width=0.6\textwidth]{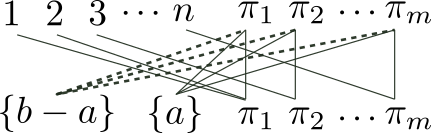}
\caption{The bipartite graph corresponds to $\bar{v}$. The top nodes of this graph form the ground set for $\bar{v}$, indexed by $[n]$ and the $m$ parts of the partition $\pi$. The bottom nodes of this graph are indexed by the $m$ parts of the partition $\pi$, together with two special nodes with labels $\{a\}$ and $\{b-a\}$, respectively. Solid edges all have weights $a$ and either connect $\pi_i$ (top) and $\pi_i$ (bottom), $\pi_i$ (top) and $\{a\}$ (bottom), or connect $\pi_i$ (bottom) to its corresponding elements in $[n]$ (top). The dashed edges all have weights $b-a$ and connect $\{b-a\}$ (bottom) to $\pi_i$ (top).}\label{fig:oxs} 
\end{figure}
Under the condition $0 < a < b < 2a$, one finds that
\begin{align*}
v(\pi) &= ma\\
v(\pi \cup I) &= 
\left\{ 
\begin{array}{ccc}
(m+1)a & \mbox{ if } & I \subseteq \pi_r \mbox{ for some } r \in \{1,\dots,m\} \\
ma + b & \mbox{ else.}  &
\end{array}\right. \, \forall \, I \subseteq [n].
\end{align*}
Therefore, $v^{\pi,a,b} = \partial^{[n]}v$. Since $v$ is OXS, $v^{\pi,a,b} \in \GS_n$. 
\end{proof}

By the proof of Lemma \ref{lem:in.gs}, a partition $\pi$ of $[n]$ into $r$ parts gives rise to a partition valuation $v^\pi: 2^{[n]} \to \R_{\geq 0}$ that is the endowment by $[n]$ of an OXS valuation in the larger set $[n+r]$, that is, $v^\pi$ is in the EAV class on ground set $[n+r]$. However, when the ground set is $[n]$, most partition valuations cannot be generated from merging and endowment of functions in $\GS_n$. 

\begin{lemma}\label{lem:clique.irr}
Let $n \geq 4$, $0 < a < b < 2a$. Suppose $\pi$ is a set partition of $[n]$ into $m$ parts such that each part has cardinality at least 2. Then $v^{\mathcal{S},a,b}$ is irreducible. In particular, it is a minimal generator for $\GS_n$ under merging and endowment.
\end{lemma}
\begin{proof}
For a function $u: 2^{[n]} \to \R$, write $u_i$ for $u(\{i\})$. Suppose that $v = u \ast u'$ for some $u,u' \in \GS_n$. Assume without loss of generality that $v_1 = u_1 = a$. For any $i \neq 1$, 
$$ b \geq v(\{1,i\}) \geq u_i +  u'_i = a + u'_i, $$
so $u'_i \leq b-a < a$. But 
$$ v_i = \max(u_i, u'_i)) = a, $$
therefore $u_i = a$ for all $i \in [n]$. By the assumption on $\pi$, for each $i \in [n]$, there exists some $j \in [n]$ such that $i$ and $j$ belong to the same partition, so
$$ v(\{i,j\}) = a \geq u_j+ u'_i = a + u'_i. $$
So $u'_i = 0$ for all $i \in [n]$. Now, fix any subset $I \subset [n]$ and $j \notin I$. Since $u' \in \GS_n$, 
$$ u'(I \cup \{j\}) \leq u'(I) + u'_j = u'(I). $$
But $u'$ is a valuation, so $u'(I \cup \{j\}) \geq u'(I)$. Thus $u'(I \cup \{j\}) = u'(I)$ for all $I \subset [n]$, so in particular $u'(I \cup \{j\}) = u'(\emptyset) = 0$. That is, $u'$ is the zero function. So $v = u$. Thus $v$ is irreducible. Since $v$ is defined on ground set $[n]$, it is in $\mathcal{G}_n$ by Corollary \ref{rem:irr}.
\end{proof}

\begin{lemma}\label{lem:notin.mbv}
Let $n \geq 4$, $0 < a < b < 2a$. Suppose $\pi$ is a set partition of $[n]$ into $m$ parts such that each part has cardinality at least 2. Then $v^{\pi,a,b}$ is not the weighted rank of some matroid $\E$ on $[n]$. 
\end{lemma}
\begin{proof}
Write $v$ for $v^{\pi,a,b}$. Suppose for contradiction that $v = \rho^w$ for some matroid $\E$ on $[n]$. For any pair $i,j \in [n]$, 
$$v(\{i,j\}) \leq b < 2a =  v_i + v_j.$$
Thus $\E$ is a matroid of rank one. Since $v_i = a = w_i> 0$ for all $i$, this matroid is loopless. Therefore, $\E$ must be the uniform matroid $U_{1,[n]}$. But by the assumption on $\pi$, there exists some $i,j \in [n]$ such that $v(\{i,j\}) = b \neq \max(w_i,w_j) = a$. So $v \neq \rho^w$, and this is the desired contradiction.
\end{proof}

\begin{proof}[Proof of Theorem \ref{thm:main} for $n \geq 4$]
For each pair of real numbers $a,b \in \R$ such that $0 < a < b < 2a$, and for $\pi = \{\{1,2\},\{3,4,\dots,n\}\}$, consider the partition valuation $v^{\pi,a,b}$. By Lemma \ref{lem:clique.irr}, it is in $\mathcal{G}_n$, the set of minimal generators of $\GS_n$ under merging and endowment. By Lemma~\ref{lem:notin.mbv}, it is not a weighted matroid rank function. Thus the set of minimal generators for $\GS_n$ is strictly larger than the set of generators of $\MBV_{n,n}$. So $\MBV_{n,n} \subsetneq \GS_n$.
\end{proof}

\section{Proof of Theorem \ref{thm:ce.irreducible}}

\subsection{A recipe for minimal generators}\label{subsec:geometry}
This section contains two technical results which are critical to the proofs of Theorems \ref{thm:ce.irreducible} and \ref{thm:main2} and yield important insights on the geometry of merging. Their proofs rely on tools from discrete convex analysis, and are presented in the appendix. The first result, Lemma \ref{lem:gs.mixed}, gives a sufficient condition for a valuation $u \in \GS_n$ to be irreducible.  

\begin{definition}[$M$-irreducible polytopes]
Let $P \subseteq [0,1]^n$ be an $M^\natural$- convex lattice polytope. Say that $P$ is $M$-irreducible if 
$$ P = (P^1+P^2) \cap [0,1]^n $$
for $M^\natural$ lattice polytopes $P^1,P^2 \subseteq [0,1]^n$ implies that either $P = P^1$ or $P = P^2$.
\end{definition}

\begin{lemma}[Geometry of merging]\label{lem:gs.mixed}
Let $u,v \in \GS_n$. For each face $F$ of $\Delta_{u \ast v}$, there exist some faces $F_u$ of $\Delta_u$ and $F_v$ of $\Delta_v$ such that
\begin{equation}\label{eqn:F.p}
F = (F_u + F_v) \cap [0,1]^n.
\end{equation}
In particular, if $\Delta_{u}$ has a full-dimensional face $F$ that is $M$-irreducible, then $u$ is in $\mathcal{G}_n$. 
\end{lemma}



\begin{remark}\label{rem:recipe}
Lemma \ref{lem:gs.mixed} gives the following method for constructing minimal generators of $\GS_n$. First, construct a full-dimensional $M$-irreducible polytope $P$. Second, embed this polytope $P$ into a suitable regular subdivision $\Delta$. Such a subdivision must satisfy three criterion: (a) $\Delta$ has $P$ as a face, (b) $\Delta = \Delta_u$ for some valuation $u: 2^{[n]} \to \R_{\geq 0}$, and (c) the edges of $\Delta$ are parallel to one of the vectors in $\{e_i-e_j,e_i\}$. By Theorem \ref{thm:murota}, the edge condition guarantees that $u \in \GS_n$. And since $\Delta_u$ has $P$, by Lemma \ref{lem:gs.mixed}, $u \in \mathcal{G}_n$, that is, it is a minimal generator of $\GS_n$ under merging and endowment.
\end{remark}
To carry out this program, we need a way to construct full-dimensional $M$-irreducible polytopes. The second main result in this section, Proposition \ref{prop:mn.irr}, precisely provides such a recipe. The proof is given in the Appendix. 

For $P \subset \R^n$ an $M^\natural$ lattice polytope, define $\rho_P: 2^{[n]} \to \R_{\geq 0}$ given by
\begin{equation}\label{eqn:rho.P}
\rho_P(I) = \max\{\sum_{i \in I}x_i: x \in P\}.
\end{equation}
\begin{proposition}\label{prop:mn.irr}
Let $P \subset [0,1]^n$ be an $M^\natural$ lattice polytope. If $\rho_P$ defined via \eqref{eqn:rho.P} is the rank function of an irreducible matroid, then $P$ is $M$-irreducible. 
\end{proposition}

\subsection{Proof of Theorem \ref{thm:ce.irreducible}}

\begin{proof}[Proof of Theorem \ref{thm:ce.irreducible}]
Suppose that $\mathcal{E} = ([n],\mathcal{I})$ is irreducible. Without loss of generality, assume that $\E$ has no loop. Let $P$ be the independence polytope of $\mathcal{E}$
$$ P = \conv\{e_I: I \in \mathcal{I}\}. $$
Since $\mathcal{E}$ has no loop, $P$ is full-dimensional. Note that $P$ is a face of $\Delta_{\rho^w}$, namely, it is the convex hull of $a \in \{0,1\}^n$ that achieves the maximum in $f_{\rho^w}(w)$ (cf. \eqref{eqn:f}). Furthermore, $\rho_P$ equals the rank function of $\E$. Since $\E$ is irreducible, by Proposition \ref{prop:mn.irr}, $P$ is $M$-irreducible. Now suppose that $\rho^w = u^1 \ast u^2$ for some $u^1,u^2 \in \GS_n$. By Lemma \ref{lem:gs.mixed}, 
$$P = (P^1 + P^2) \cap [0,1]^n$$ 
for some faces $P^i \in \Delta_{u^i}$, $i = 1, 2$. Since $P$ is $M$-irreducible, either $P^1$ or $P^2$ must be equal to $P$. Without loss of generality, assume that this is $P^1$. Then $P^2 = \{e_\emptyset\} \in \Delta_{u^2}$. Thus, for all $S \in \mathcal{I}$
 \begin{equation}\label{eqn:rho.I}
\rho^w(S) = \sum_{i \in S}w_i = u^1(S) + v^1(\emptyset) = u^1(S).
\end{equation}
Now consider a set $S \subseteq [n]$ where $S \notin \mathcal{I}$. Since $u^2$ is a valuation, $u^1 \leq u^1 \ast u^2 = \rho^w$. Then
\begin{align*}
\rho^w(S) &= \max_{I \in \mathcal{I}, I \subset S}\rho^w(I) & \mbox{ by definition} \\
&= \max_{I \in \mathcal{I}, I \subset S}u^1(I) & \mbox{ by \eqref{eqn:rho.I}} \\
&\leq u^1(S) & \mbox{ since $u^1$ is a valuation.}
\end{align*}
Therefore, $\rho^w(S) = u^1(S)$ for all $S \subseteq [n]$, so $\rho^w$ is an irreducible function. For the converse, suppose that $\mathcal{E} = ([n],\mathcal{I})$ is reducible, so $\mathcal{E} = \mathcal{E}^1 \vee \mathcal{E}^2$, for matroids $\mathcal{E}^1$ (resp. $\E^2$) defined on ground set $E^1$ (resp. $E^2$) with independence set $\mathcal{I}^1$ (resp. $\mathcal{I}^2$). 
Let $w^1 \in \R^{E^1}$ be the restriction of $w$ to $E^1$, and $w^2 \in \R^{E^2}$ be the restriction of $w$ to $E^2$. For $S \subseteq [n]$, 
\begin{align*}
\rho^{w^1} \ast \rho^{w^2}(S) &= \max_{T \subseteq S}\rho^{w^1}(T) + \rho^{w^2}(S-T) \\
&= \max_{T \subseteq S}\rho^{w^1}(T) + \rho^{w^2}(S-T) \\
&= \max_{T \subseteq S}\left(\max_{T^1 \in \mathcal{I}^1, T^1 \subseteq T}\sum_{i \in T^1}w_i + \max_{T^2 \in \mathcal{I}^2, S^2 \subseteq S-T}\sum_{i \in S^2}w_i\right) \\
&= \max_{T \subseteq S}\left(\max_{T^1 \in \mathcal{I}^1, T^1 \subseteq T}\max_{T^2 \in \mathcal{I}^2, S^2 \subseteq S-T}\sum_{i \in T^1 \cup S^2}w_i\right) \\
&= \max_{T' \subseteq S, T' \in \mathcal{I}} \sum_{i \in T'}w_i \hspace{1cm} \mbox{ by definition of matroid union} \\
&= \rho^w(S).
\end{align*}
Thus $\rho^w = \rho^{w^1} \ast \rho^{w^2}.$ Since $\mathcal{E}^1 \neq \mathcal{E}^2$, $\rho^{w^1}, \rho^{w^2} \neq \rho^w$, and thus $\rho^w$ is a reducible function, as needed.
\end{proof}

\section{Proof of Theorem \ref{thm:main2}}\label{sec:main2}

In this section, we carry out the program outlined in Remark \ref{rem:recipe} to construct another large family of irreducibles for $n \geq 6$ that cannot be obtained as partition valuations nor weighted matroid ranks, thereby proves Theorem \ref{thm:main2}. Our strategy is illustrated in Figure \ref{fig:illustrate}. 

\begin{figure}[h!]
\includegraphics[width=0.4\textwidth]{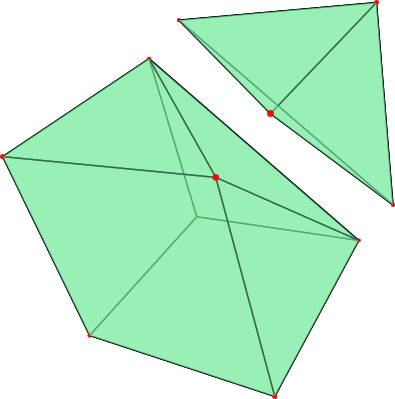} \quad 
\includegraphics[width=0.4\textwidth]{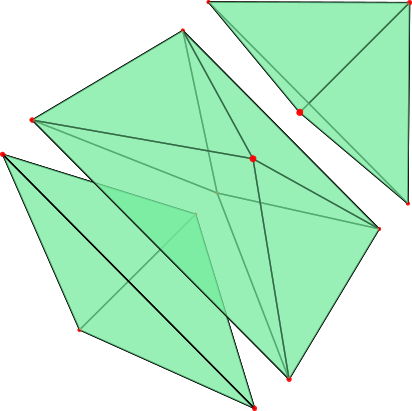}
\caption{Proof idea of Theorem \ref{thm:main2}. Figure not to scale and is for illustration purpose only. The LHS depicts the regular subdivision $\Delta_{\rho^w}$ for the weighted rank of an irreducible matroid. We modify $\rho^w$ to split the origin (bottom left) away while leaving the other faces of the regular subdivision the same. This gives a new regular subdivision $\Delta_v$ (RHS), where the independence polytope $P$ is split into two $M^\natural$ lattice polytopes: the standard simplex on $[n]$, and the complement $Q$. By Proposition \ref{prop:mn.irr}, $Q$ is $M$-irreducible, and this face witnesses the irreducibility of $v$ in $\GS_n$ by Lemma \ref{lem:gs.mixed}.}\label{fig:illustrate}
\end{figure}

Let $\rho^w$ be the weighted matroid rank of a loopless irreducible matroid $\E$ of rank at least two. Let $\mathbf{1}: \{0,1\}^n \to \{1\}$ be the all-one function, and $\mathbf{1}_{\emptyset}: \{0,1\}^n \to \{0,1\}$ be the indicator of the origin, that is
$$ \mathbf{1}_{\emptyset}(\emptyset) = 1, \quad \mathbf{1}_{\emptyset}(I) = 0 \mbox{ for all } I \subseteq [n], I \neq \emptyset. $$
\begin{proposition}\label{prop:modify}
For any $w \in \R^n_{>0}$ that is not a constant multiple of the all-one vector, and any $c > 0$, the valuation
\begin{equation}\label{eqn:v.rho} 
v := \rho^w + c \cdot (\mathbf{1}-\mathbf{1}_{\emptyset})
\end{equation}
satisfies the following
\begin{enumerate}
  \item[(i)] $v \in \GS_n$ and $v$ is irreducible
  \item[(ii)] $v \notin \MBV_{n,n}$ and $v$ is not a partition valuation.
\end{enumerate}
\end{proposition}
\begin{proof}
By definition, $v(\emptyset) = \rho^w(\emptyset) = 0$, and for non-emptyset $S,T$, $S \subseteq T$ implies
$v(S) = \rho(S) + c \leq \rho(T) + c = v(T)$. Thus $v$ is a valuation. 
To see the difference between $\Delta_{\rho^w}$ and $\Delta_v$, it is easier to consider $\Delta_{v'}$ where $v' := \rho^w - c\cdot\mathbf{1}_\emptyset = v - c\cdot\mathbf{1}$. Since regular subdivisions is invariant under addition by constant, $\Delta_v = \Delta_{v'}$, pointwise. Now, $v'(I) = \rho^w(I)$ for all $I \subseteq [n]$, except for the emptyset, where $v'(\emptyset) = -c < 0$. Therefore, $\Delta_{v'} = \Delta_v$ is obtained from $\Delta_{\rho^w}$ by splitting the origin away from the independence polytope $P = \conv(\{e_I: I \in \mathcal{I}\})$ of $\E$, creating two faces (see Figure \ref{fig:illustrate}). These two faces are the simplex at the origin
$$P_0 = \conv(\{0, e_i: i \in [n]\})$$ 
and its complement
$$ Q = \conv(\{e_I: I \in \mathcal{I}, I \neq \emptyset\}).$$ 
Thus compared to $\Delta_{\rho^w}$, $\Delta_v$ has extra edges of the form $e_i - e_j$. Since $\rho^w \in \GS_n$, by Theorem \ref{thm:murota}, the edges of $\Delta_{\rho^w}$ are parallel to $\{e_i-e_j\}$. Therefore, the edges of $\Delta_v$ are also parallel to $\{e_i-e_j\}$, so by Theorem \ref{thm:murota}, $v \in \GS_n$ and $Q$ is an $M^\natural$ lattice polytope. By definition, $\rho_Q = \rho$. Since $\rho$ is irreducible, by Proposition \ref{prop:mn.irr}, $Q$ is $M$-irreducible. Since $\E$ is loopless and has rank at least 2, $Q$ is full-dimensional. Now, suppose that $v = v^1 \ast v^2$ for some $v^1,v^2 \in \GS_n$. By Lemma \ref{lem:gs.mixed},
$$ Q = (Q^1 + Q^2) \cap [0,1]^n $$
for some faces $Q^1$ of $\Delta_{\omega}$ and $Q^2$ of $\Delta_{\tau}$. Since $Q$ is $M$-irreducible, it is equal to either $Q^1$ or $Q^2$. Suppose without loss of generality that $Q^1 = Q$. Thus $Q^2$ is the origin, and
$$ v(I) = v^1(I) + v^2(\emptyset) = v^1(I) \mbox{ for all } I \in \mathcal{I}, I \neq \emptyset.$$
Since $v$ and $v^1$ are both valuations, $v(\emptyset) = \omega(\emptyset)$, so the above holds for all $I \in \mathcal{I}$. Now let $S \subseteq [n]$. Since $v^1,v^2 \geq 0$,
\begin{align*}
v^1(S) \leq v^1 \ast v^2(S) = v(S) &= \max\{v(I): I \subseteq S, I \in \mathcal{I}\} \\
&= \max\{v^1(I): I \subseteq S, I \in \mathcal{I}\} \leq v^1(S).
\end{align*}
So $v^1(S) = v(S)$ for all $S \subseteq [n]$. Thus $v$ is irreducible. For the second claim, since $w$ has at least two different entries, $v$ have at least $3$ different non-zero values, so it is not a partition valuation.
We now prove that $v \notin \MBV_{n,n}$. Since $v$ is irreducible and has ground set $[n]$, by Corollary \ref{rem:irr}, it is sufficient to show that $v$ is not the weighted rank of some irreducible matroid $\E'$, or some partition valuation. In the first case, suppose for contradiction that $v$ is the weighted rank of some matroid $\E'$. Then the largest face of $\Delta_v$ that contains the origin must equals the independence polytope of $\E'$. But this face is the simplex $P_0$, therefore, $\E'$ is the uniform matroid of rank 1 on $n$. But $\E$ has rank at least two, there exist some independent set $I = \{i,j\}$ of $\E$ of rank two. On this set,
$$ v(ij) = \max(w_i+c,w_j+c)$$
By definition of $v$ in \eqref{eqn:v.rho}, we have
$$ v(ij) = w_i + w_j + c.$$
Since $w \in \R^n_{>0}$, these two quantities cannot be equal, a contradiction. Therefore, $v$ is not a weighted matroid rank, as needed.
\end{proof}

\begin{proof}[Proof of Theorem \ref{thm:main2}]
For $n = 6$, the graphical matroid on the complete graph on four vertices $M(K_4)$ is a loopless, irreducible binary matroid of rank 3 \cite[p405]{oxley2006matroid}. By Lov\'{a}sz and Recski \cite{lovasz1973sum}, a connected binary matroid is irreducible if and only if its one-element deletions are also connected \cite[p405]{oxley2006matroid}. These properties and being loopless are preserved under parallel extensions, therefore, there is at least one loopless irreducible matroid for each $n \geq 6$. For each $n$, let $\E$ be such a matroid. For each weighted rank $\rho^w$ of $\E$, construct $v$ as \eqref{eqn:v.rho}. Proposition \ref{prop:modify} applied to $v$ implies Theorem \ref{thm:main2}.
\end{proof}
This proof requires $n \geq 6$ since it starts with a loopless irreducible matroid of rank at least two, and that the smallest such matroid is for $n = 6$. 
\begin{lemma}\label{lem:no.irr.5}
Let $n \leq 5$ and $\E$ be an irreducible matroid without loop on $[n]$. Then $\E$ is the uniform matroid on $[n]$ of rank 1. 
\end{lemma}
\begin{proof}
Let $U_{r,S}$ denote the uniform matroid of rank $r$ on the ground set $S$. It is clear that $U_{1,[n]}$ is irreducible. Now suppose that $\E$ is another irreducible matroid. Then $\E$ must be connected, and its one-element deletions must also be connected \cite{recski1989some}. Exhaustive enumeration using the database of matroid \cite{Matsumoto2012,matroid-db} and the software Sage \cite{sagemath} shows that the only matroids $\E$ on $n \leq 5$ elements that satisfies these properties are $U_{2,[4]}, U_{2,[5]}, U_{[3],5}$, and the matroid $\E_2$ of rank 2 on 5 elements, obtained by taking $U_{2,[5]}$ and excluding $\{1,2\}$ from its set of bases. By direct computations, 
\begin{align*}
U_{2,[4]} = U_{1,[4]} \vee U_{1,[3]}, & \quad U_{2,[5]} = U_{1,[5]} \vee U_{1,[4]}, \\
U_{3,[5]} = U_{2,[5]} \vee U_{1,[4]}, & \quad \E_2 = U_{1,[5]} \vee U_{1,\{3,4,5\}}.
\end{align*}
Therefore, none of these matroids are irreducible. 
\end{proof}

\section{Summary and Open Questions}\label{sec:summary}
The matroid-based valuation (MBV) conjecture states that all gross substitutes valuations on at most $n$ items can be generated from repeatedly merging and taking the endowments of weighted rank of matroids defined on subsets of at most $m$ items. For finite $n$, equality can be achieved between these two classes if and only if $m = n$. In this paper we proved that this finite version of the MBV conjecture holds when the number of items $n$ is at most 3, and fails for all $n \geq 4$. Our proof is constructive: for small $n$ it gives an explicit decomposition based on the geometry of gross substitute valuations, for large $n$ it gives a large family of gross substitutes valuations that cannot be obtained through merging and endowments of weighted ranks. We also showed that some matroid-based valuations themselves are also the merge of other matroid-based valuations, and went on to characterize all the extreme elements in this class. These are precisely the weighted ranks of irreducible matroids. 

Our paper leaves three interesting open questions at the intersection of economics and matroid theory. First, our results indicate that merging and endowment alone are not rich enough operations to generate all gross substitutes valuations from a small subclass. Hatfield and Milgrom \cite{hatfield2005matching} showed that a large number of gross substitutes valuations that arise in economic applications are endowed assignment valuations (EAV), a much smaller class of functions obtained by merging and endowments of unit demand valuations. It would be interesting to systematically generalize this class even further, not with the goal of generating all gross substitutes valuations, but to generate a larger and useful subclass with tractable representations. For example, the partition valuations contain unit demand valuations with constant weights as a special case. Could they be generalized to a weighted version so that they subsume the EAV class? What would be their economics interpretations?

The second open direction is: what are the class of all irreducible gross substitutes valuations on $[n]$? Our results show that characterizing this class is fundamentally tied to the long-standing question of Welsh on characterizing irreducible matroids. We remark that endowment generalizes matroid contraction. However, it is not immediately obvious that the endowment of a weighted rank of some matroid $\E$ is the weighted rank of the contraction of $\E$. Is this true in general? If not, what are the gross substitutes valuations which are irreducible under merging but can be obtained as the contraction of the weighted rank of an irreducible matroid upstairs? These questions are fundamental to tackling the general MBV conjecture, and we hope that they will fuel new developments at the intersection of theoretical economics and matroid theory.    

\bibliography{nsf-refs.bib}
\bibliographystyle{alpha}

\appendix

\section{Proofs of results in Section \ref{subsec:geometry}}
Following \cite{murota2003discrete}, for $q \in \R^n$, $u[-q]: \mathbb{Z}^n \to \R \cup \{\pm \infty\}$ is the function $u[-q](a) := u(a) - \langle q,a\rangle$. 
For $u$ concave, $\sigma \subset \mathbb{Z}^n$ is a cell of $\Delta_{u}$ supported by some vector $q \in \R^n$ means
$\sigma = \arg\max u[q] = \arg\min (-u)[-q]$.    
The indicator function of a lattice polytope $Q$ is $\delta_Q: \mathbb{Z}^n \to \R \cup \{+\infty\}$, defined as
$$ \delta_Q(x) = 0 \mbox{ if } x \in Q \cap \mathbb{Z}^n, \quad \delta_Q(x) = +\infty \mbox{ else }. $$
The infimum convolution of two functions $u,v: \mathbb{Z}^n \to \R \cup \{+\infty\}$ over $\mathbb{Z}$ is the function $u \infconv v: \mathbb{Z}^n \to \R \cup \{+\infty\}$ given by
$$ u \infconv v(a) = \inf\{u(b) + v(b'): b+b' = a, b,b' \in \mathbb{Z}^n\}. $$

\subsection{Proof of Lemma \ref{lem:gs.mixed}}\label{proof:lem:gs.mixed} 
Since $u,v$ are gross substitute valuations, $(-u),(-v): \mathbb{Z}^n \to \R \cup \{+\infty\}$ are $M^\natural$-convex functions \cite[\S 6.8]{murota2003discrete}. By a definition chase, we have
\begin{equation}\label{eqn:u.ast.v.infconv}
-(u \ast v) = (-u) \infconv (-v) + \delta_{[0,1]^n}.
\end{equation}
Now, let $F$ be a face of $\Delta_{u \ast v}$. By definition, $F = \conv(\sigma)$ where $\sigma \subset \{0,1\}^n$ is a cell of $\Delta_{u \ast v}$ supported by some vector $q \in \R^n$. That is,
\begin{equation}\label{eqn:sigma.star}
\sigma = \arg\min (-(u \star v))[-q].
\end{equation}
By \eqref{eqn:u.ast.v.infconv}, 
$$ (-(u \ast v))[-q] = (-u) \infconv (-v)[-q] + \delta_{[0,1]^n}. $$
Fix a point $a^\ast \in \sigma$. Then
\begin{align*}
(-(u \ast v))[-q](a^\ast) &= (-u) \infconv (-v)[-q](a^\ast) + \delta_{[0,1]^n}(a^\ast) \quad \mbox{ since } \sigma \subset \{0,1\}^n \\
&\leq (-(u \ast v))[-q](a) \quad \mbox{ for all } a \in \mathbb{Z}^n \mbox{ by \eqref{eqn:sigma.star}} \\
&= (-u) \infconv (-v)[-q](a) + \delta_{[0,1]^n}(a) \quad \mbox{ by \eqref{eqn:u.ast.v.infconv}.}
\end{align*}
Apply the $M$-convex intersection theorem \cite[Theorem 8.17]{murota2003discrete} for $f_1 = (-u) \infconv (-v)[-q]$ and $f_2 = \delta_{[0,1]^n}$, we have that there exists some $p^\ast \in \R^n$ such that
$$ \sigma = \arg\min (-(u \ast v))[-q] = \arg\min f_1[-p^\ast] \cap \arg\min f_2[+p^\ast]. $$ 
Since $f_1[-p^\ast] = (-u) \infconv (-v)[-q-p^\ast]$, by \cite[Proposition 8.41]{murota2003discrete}, we have
$$ \arg\min f_1[-p^\ast] = \arg\min (-u)[-q-p^\ast] + \arg\min (-v)[-q-p^\ast] = \sigma^1 + \sigma^2, $$
where $\sigma^1$ is the cell $\arg\min (-u)[-q-p^\ast]$ of $\Delta_{u}$, while $\sigma^2$ is the cell $\arg\min (-v)[-q-p^\ast]$ of $\Delta_v$. Since $f_2$ is the indicator function of the cube, we can replace $\arg\min f_2[+p^\ast] \subseteq \{0,1\}^n$ by $\{0,1\}^n$. Thus we have
$$ F = \conv(\sigma) = \conv((\sigma^1 + \sigma^2) \cap \{0,1\}^n). $$ 
Since $u,v,u \star v \in \GS_n$, by Theorem \ref{thm:murota}, $\conv(\sigma^1), \conv(\sigma^2)$ and $\conv(\sigma)$ are $M^\natural$ lattice polytopes. In particular, taking convex hull, intersections with the lattice and taking Minkowski addition commute. Therefore,
\begin{align*}
F &= \conv((\sigma^1 + \sigma^2) \cap \{0,1\}^n) \\
&= \conv(\sigma^1 + \sigma^2) \cap [0,1]^n \quad \mbox{since }\conv(\sigma^1+\sigma^2) \mbox{ is $M^\natural$} \\
&= (\conv(\sigma^1) + \conv(\sigma^2)) \cap [0,1]^n \mbox{since }\conv(\sigma^1),\conv(\sigma^2) \mbox{ are $M^\natural$} \\
&= (F^1 + F^2) \cap [0,1]^n,
\end{align*}
where $F^1$ is a face of $\Delta_u$ and $F^2$ is a face of $\Delta_v$. This completes the proof. \qed

\begin{lemma}\label{lem:is.rank}
Suppose $P \subset [0,1]^n$ is an $M^\natural$ lattice polytope. Then $\rho_P$ defined in \eqref{eqn:rho.P} is the rank function of a matroid. 
\end{lemma}
\begin{proof}
Let $I \subseteq [n]$. Since $P$ is a lattice polytope, $\rho_P(I) \in \mathbb{N}$. Since $P \subset [0,1]^n$, $\rho_P(I) \leq |I|$ and $\rho_P$ is monotone non-decreasing, that is, $J \subseteq I \Rightarrow \rho_P(J) \leq \rho_P(I)$. Finally, $P$ is an $M^\natural$-convex set, $\rho_P$ is submodular. So $\rho_P$ satisfies the rank axioms of a matroid \cite[\S 1.3]{oxley2006matroid}. 
\end{proof}

\subsection{Proof of Proposition \ref{prop:mn.irr}}
Suppose for contradiction that $P$ is not $M$-irreducible, that is,
\begin{equation}\label{eqn:assume.P}
 P = (P^1 + P^2) \cap [0,1]^n 
\end{equation}
for some $M^\natural$ polytopes $P^1,P^2 \subset [0,1]^n$. 
By Lemma \ref{lem:is.rank}, $\rho_{P^1}, \rho_{P^2}$ are the rank functions of some matroids. For a polytope $Q$, let $\delta_Q: \R^n \to \R \cup \{+\infty\}$ be the indicator function of $Q$, that is
$$ \delta_Q(x) = 0 \mbox{ if } x \in Q, \quad \delta_Q(x) = +\infty \mbox{ if } x \notin Q. $$
By \cite[p80]{murota2003discrete}, \eqref{eqn:assume.P} implies
$$ \delta_P = (\delta_{P^1} \square \delta_{P^2}) + \delta_{[0,1]^n} $$
Now we take the Legendre-Fenchel transform of both sides. By the integer convolution formula \cite[Theorem 8.36]{murota2003discrete}
$$ \delta_P^\vee = (\delta^\vee_{P^1} + \delta^\vee_{P^2}) \square \delta^\vee_{[0,1]^n} $$
Thus, for $I \subseteq [n]$,
\begin{align*}
\rho_P(I) = \delta_P^\vee(e_I) &= \inf\{(\delta^\vee_{P^1} + \delta^\vee_{P^2})(x) + \sum_{i \in I}y_i: x+y = e_I, x,y \in \mathbb{Z}^n\} \\
&= \min\{(\delta^\vee_{P^1} + \delta^\vee_{P^2})(J) + |I-J|: J \subseteq I\} \\
&= \min_{J \subseteq I}(\rho_{P^1}(J) + \rho_{P^2}(J) + |I-J|). 
\end{align*}
By \cite[Corollary 42.1a]{schrijverA}, this implies $\E = \E^1 \vee \E^2$, a contradiction on the irreducibility of $\E$. This concludes the proof. \qed

\section{Gross substitutes valuations for $n = 3$}
Figures \ref{fig:type0} to \ref{fig:type.last} accompany the proof of Theorem \ref{thm:main} for $n = 3$. For each figure, the LHS gives a parametrization of $v$ that is necessary to obtain this regular subdivision, the caption shows the conditions on the weights, while the RHS writes this valuation as a maximum bipartite matching. This establishes both Theorem \ref{thm:main} for the case $n = 3$ and Corollary~\ref{cor:oxs}. 

\begin{figure}[h!]
\begin{tikzpicture}[->,>=stealth',shorten >=1pt,auto,node distance=1.25cm,semithick]
\node [xshift=-4cm, yshift=-1.5cm] (delta)
    {\includegraphics[width=0.45\textwidth]{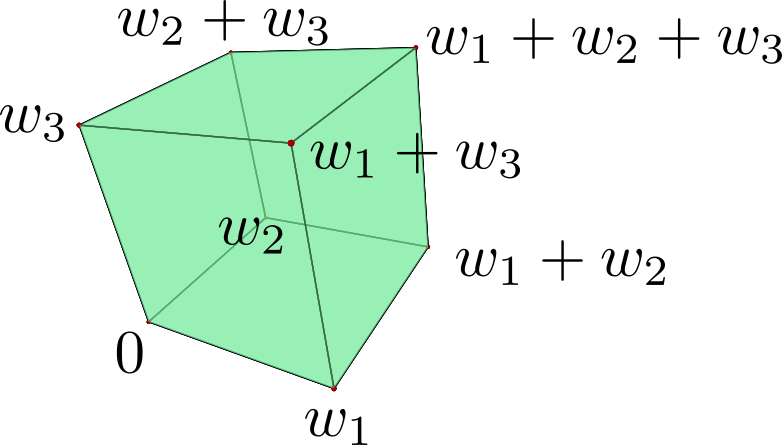} };
  \node (1) {$1$};
  \node (2) [below of=1] {$2$};
  \node (3) [below of=2] {$3$};  
  \node (A) [right of =1,xshift=2cm] {$A$};
  \node (B) [below of = A] {$B$};  
  \node (C) [below of = B] {$C$};    
  \path (1) edge node [above] {$w_1$} (A);
  \path (2) edge node [above,sloped] {$w_2$} (B);  
  \path (3) edge node [above] {$w_3$} (C);    
\end{tikzpicture} 
\caption{$0 \leq w_3 \leq w_2 \leq w_1$.}\label{fig:type0}
\end{figure}
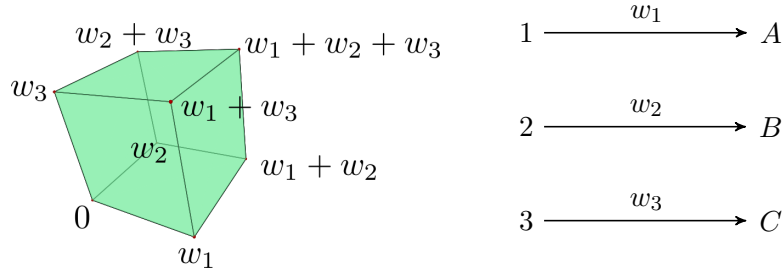

\begin{figure}[h!]
\begin{tikzpicture}[->,>=stealth',shorten >=1pt,auto,node distance=1.25cm,semithick]
\node [xshift=-4cm, yshift=-1.5cm] (delta)
    {\includegraphics[width=0.45\textwidth]{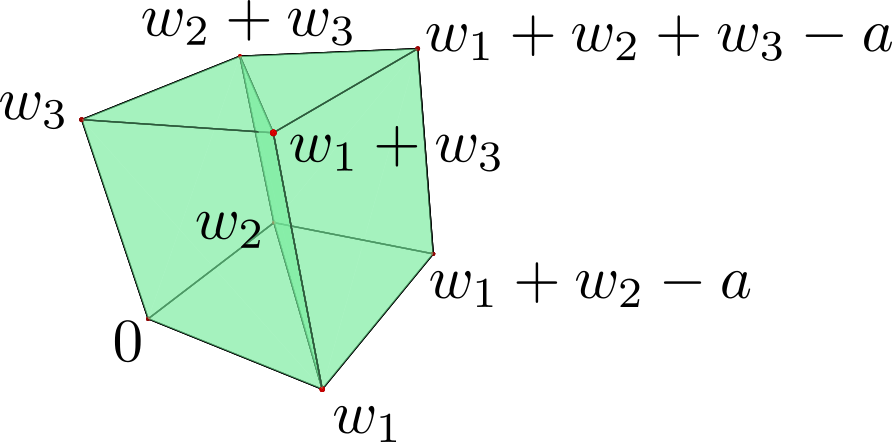} };
  \node (1) {$1$};
  \node (2) [below of=1] {$2$};
  \node (3) [below of=2] {$3$};  
  \node (A) [right of =1,xshift=2cm] {$A$};
  \node (B) [below of = A] {$B$};  
  \node (C) [below of = B] {$C$};    
  \path (1) edge node [above] {$w_1$} (A);
  \path (2) edge node [above,sloped] {$w_2$} (A);  
  \path (2) edge node [below] {$w_2-a$} (B);  
  \path (3) edge node [above] {$w_3$} (C);    
\end{tikzpicture}
\caption{$0 < w_3 \leq w_2 \leq w_1$, $a < w_2$}
\end{figure}
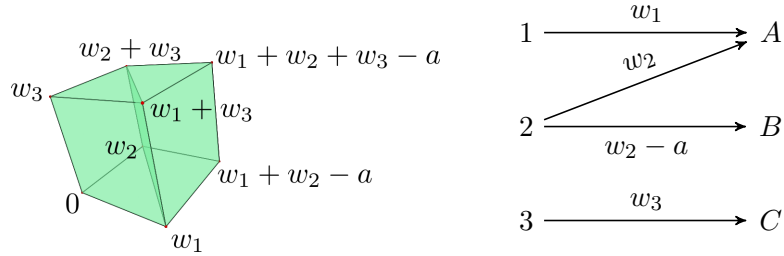

\begin{figure}[h!]
\begin{tikzpicture}[->,>=stealth',shorten >=1pt,auto,node distance=1.25cm,semithick]
\node [xshift=-4cm, yshift=-1.5cm] (delta)
    {\includegraphics[width=0.4\textwidth]{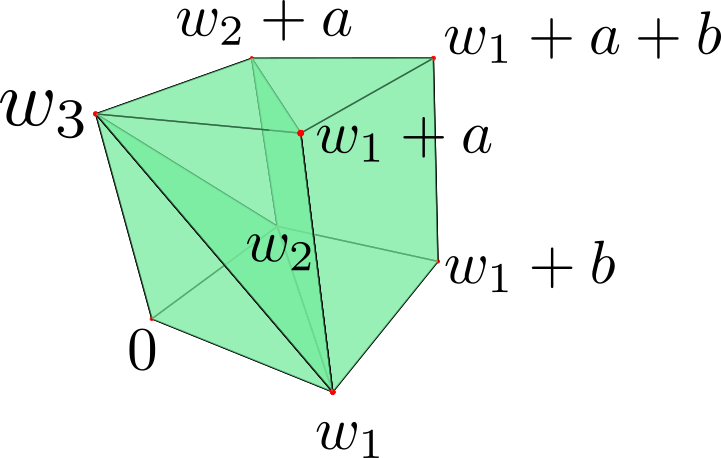}};
  \node (1) {$1$};
  \node (2) [below of=1] {$2$};
  \node (3) [below of=2] {$3$};  
  \node (A) [right of =1,xshift=2cm] {$A$};
  \node (B) [below of = A] {$B$};  
  \node (C) [below of = B] {$C$};    
  \path (1) edge node [above] {$w_1$} (A);
  \path (2) edge node [above,sloped] {$w_2$} (A); 
  \path (3) edge node [above,xshift=0.3cm,sloped] {$w_3$} (A);    
  \path (2) edge node [below,xshift=0.35cm] {$b$} (B);  
  \path (3) edge node [above] {$a$} (C);    
\end{tikzpicture}
\caption{$0 \leq a < w_3, 0 \leq b < w_2,$ $w_3 - a \neq w_2 - b$, $w_3 \leq w_2 \leq w_1$.}
\end{figure}
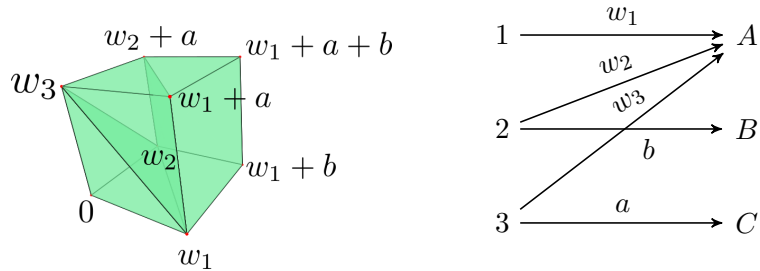

\begin{figure}[h!]
\begin{tikzpicture}[->,>=stealth',shorten >=1pt,auto,node distance=1.25cm,semithick]
\node [xshift=-4cm, yshift=-1.5cm] (delta)
    {\includegraphics[width=0.45\textwidth]{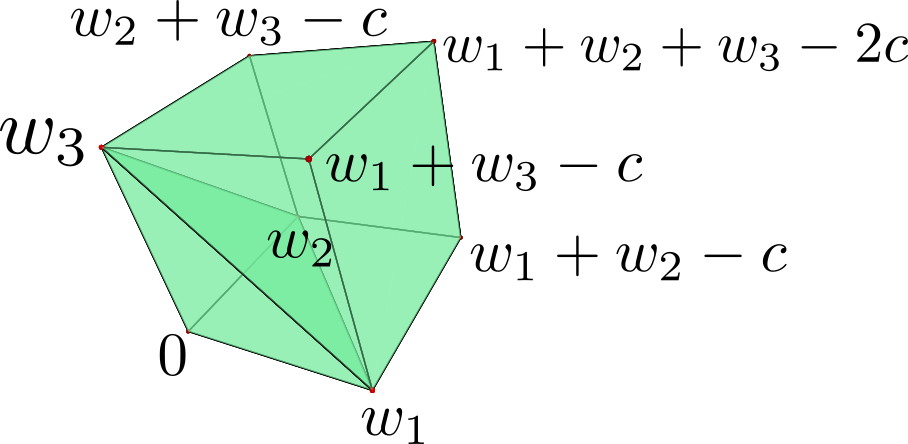}};
  \node (1) {$1$};
  \node (2) [below of=1] {$2$};
  \node (3) [below of=2] {$3$};  
  \node (A) [right of =1,xshift=2cm] {$A$};
  \node (B) [below of = A] {$B$};  
  \node (C) [below of = B] {$C$};    
  \path (1) edge node [above] {$w_1$} (A);
  \path (2) edge node [above,sloped] {$w_2$} (A); 
  \path (3) edge node [above,xshift=0.3cm,sloped] {$w_3$} (A);    
  \path (2) edge node [below,xshift=0.5cm] {$w_2-c$} (B);  
  \path (3) edge node [above] {$w_3-c$} (C);    
\end{tikzpicture}
\caption{$0 < c \leq w_3 \leq w_2 \leq w_1$.} 
\end{figure}

\begin{figure}[h!]
\begin{tikzpicture}[->,>=stealth',shorten >=1pt,auto,node distance=1.25cm,semithick]
\node [xshift=-4cm, yshift=-1.5cm] (delta)
    {\includegraphics[width=0.45\textwidth]{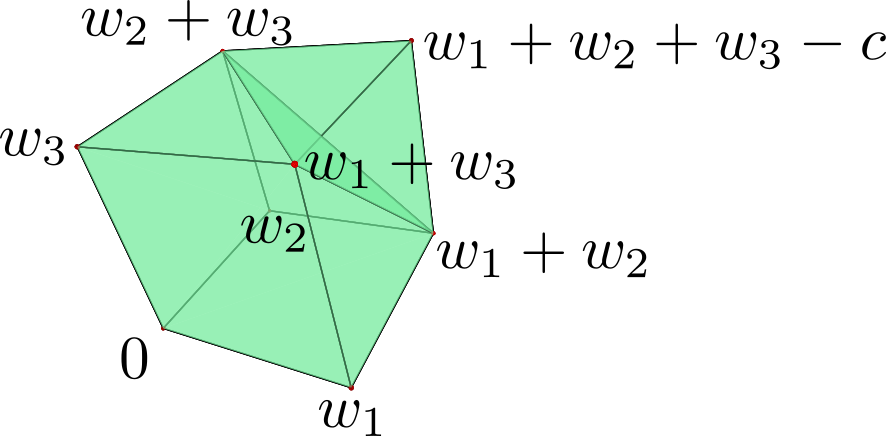}};
  \node (1) {$1$};
  \node (2) [below of=1] {$2$};
  \node (3) [below of=2] {$3$};  
  \node (A) [right of =1,xshift=2cm] {$A$};
  \node (B) [below of = A] {$B$};  
  \node (C) [below of = B] {$C$};    
  \path (1) edge node [above] {$w_1$} (A);
  \path (2) edge node [above,sloped] {$w_2$} (A); 
  \path (3) edge node [above,xshift=0.3cm,sloped]{$w_3$} (A);    
  \path (2) edge node [below,xshift=0.35cm,yshift=0.1cm] {$w_2$} (B);  
  \path (3) edge node [below] {$w_3$} (B);    
  \path (3) edge node [below] {$w_3-c$} (C);    
\end{tikzpicture}
\caption{$0 < c \leq w_3 \leq w_2 \leq w_1$}
\end{figure}

\begin{figure}[h!]
\begin{tikzpicture}[->,>=stealth',shorten >=1pt,auto,node distance=1.25cm,semithick]
\node [xshift=-4cm, yshift=-1.5cm] (delta)
    {\includegraphics[width=0.5\textwidth]{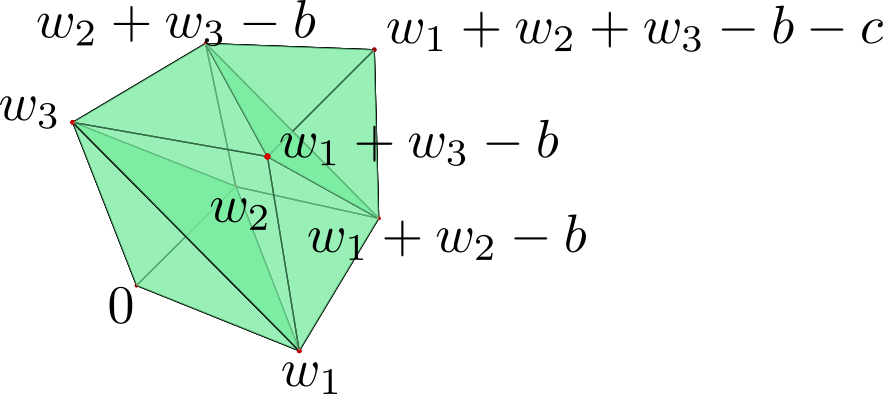}
};
  \node (1) {$1$};
  \node (2) [below of=1] {$2$};
  \node (3) [below of=2] {$3$};  
  \node (A) [right of =1,xshift=2cm] {$A$};
  \node (B) [below of = A] {$B$};  
  \node (C) [below of = B] {$C$};    
  \path (1) edge node [above] {$w_1$} (A);
  \path (2) edge node [above,sloped] {$w_2$} (A); 
  \path (3) edge node [above,xshift=0.3cm,sloped] {$w_3$} (A);    
  \path (2) edge node [below] {$w_2-b$} (B);  
  \path (3) edge node [below,sloped] {$w_3-b$} (B);    
  \path (3) edge node [below] {$w_3-c$} (C);    
\end{tikzpicture}
\caption{$0 < b < c < w_3 \leq w_2 \leq w_1$}
\end{figure}

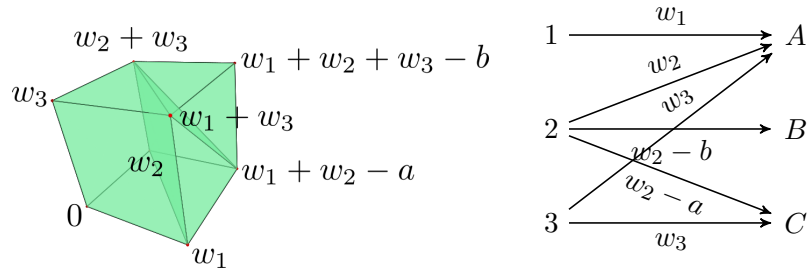
\begin{figure}[h!]
\begin{tikzpicture}[->,>=stealth',shorten >=1pt,auto,node distance=1.25cm,semithick]
\node [xshift=-4cm, yshift=-1.5cm] (delta)
    {\includegraphics[width=0.5\textwidth]{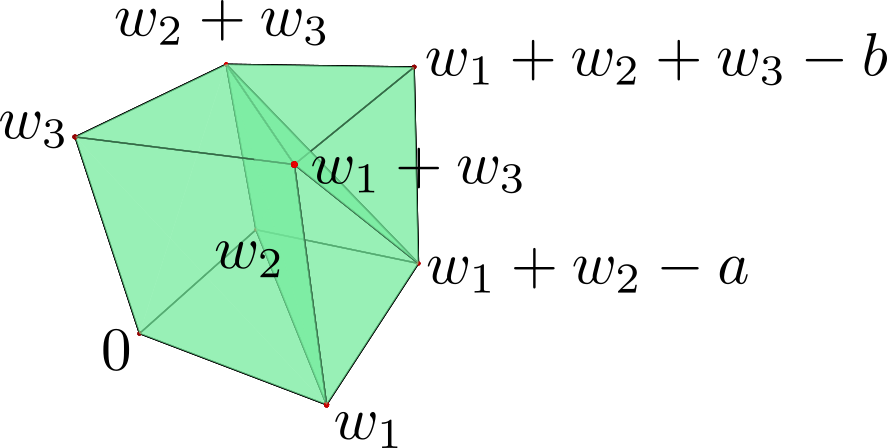}
};
  \node (1) {$1$};
  \node (2) [below of=1] {$2$};
  \node (3) [below of=2] {$3$};  
  \node (A) [right of =1,xshift=2cm] {$A$};
  \node (B) [below of = A] {$B$};  
  \node (C) [below of = B] {$C$};    
  \path (1) edge node [above] {$w_1$} (A);
  \path (2) edge node [above,sloped] {$w_2$} (A);  
  \path (3) edge node [above,xshift=0.3cm,sloped] {$w_3$} (A);    
  \path (2) edge node [below] {$w_2-b$} (B);  
  \path (2) edge node [below,sloped] {$w_2-a$} (C);    
  \path (3) edge node [below] {$w_3$} (C);    
\end{tikzpicture}
\caption{$0 < a < b$, $w_3 \leq w_2 \leq w_1$, $b \leq w_2$, $b-a<w_3$}
\end{figure}

\begin{figure}[h!]
\begin{tikzpicture}[->,>=stealth',shorten >=1pt,auto,node distance=1.25cm,semithick]
\node [xshift=-4cm, yshift=-1.5cm] (delta)
    {\includegraphics[width=0.5\textwidth]{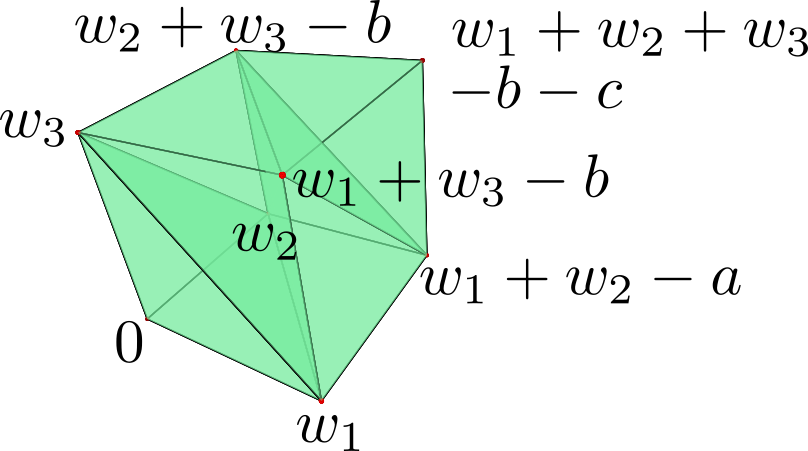}
};
  \node (1) {$1$};
  \node (2) [below of=1] {$2$};
  \node (3) [below of=2] {$3$};  
  \node (A) [right of =1,xshift=2cm] {$A$};
  \node (B) [below of = A] {$B$};  
  \node (C) [below of = B] {$C$};    
  \path (1) edge node [above] {$w_1$} (A);
  \path (2) edge node [above,sloped] {$w_2$} (A);  
  \path (3) edge node [above,xshift=0.3cm,sloped] {$w_3$} (A);    
  \path (2) edge node [below] {$w_2-a$} (B);  
  \path (3) edge node [below,sloped] {$w_3-b$} (B);    
  \path (3) edge node [below] {$w_3-b-c$} (C);    
\end{tikzpicture}
\caption{$0<a<w_2$, $0<b<b+c<w_3$, $w_3 \leq w_2 \leq w_1$}\label{fig:type.last}
\end{figure}
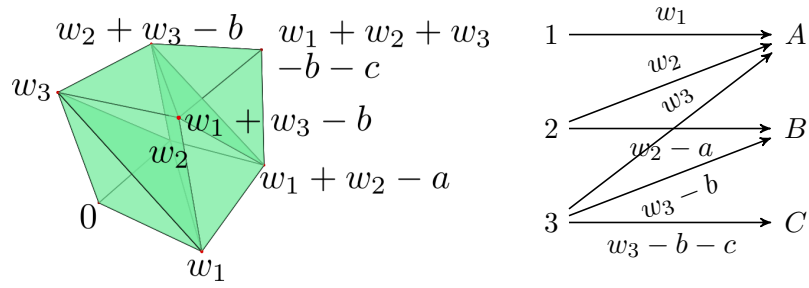

\end{document}